\documentclass[11pt,twoside]{article} 
\usepackage{bbm}
\usepackage{amsmath,amssymb}
\usepackage{slashed}
\usepackage{extarrows}
\usepackage{cancel}
\usepackage{amsmath}
\usepackage{titlesec}
\usepackage{hyperref}
\allowdisplaybreaks[4]

\makeatletter
\newcommand*{\mint}[1]{%
  \mint@l{#1}{}%
}
\newcommand*{\mint@l}[2]{%
  \@ifnextchar\limits{%
    \mint@l{#1}%
  }{%
    \@ifnextchar\nolimits{%
      \mint@l{#1}%
    }{%
      \@ifnextchar\displaylimits{%
        \mint@l{#1}%
      }{%
        \mint@s{#2}{#1}%
      }%
    }%
  }%
}
\newcommand*{\mint@s}[2]{%
  \@ifnextchar_{%
    \mint@sub{#1}{#2}%
  }{%
    \@ifnextchar^{%
      \mint@sup{#1}{#2}%
    }{%
      \mint@{#1}{#2}{}{}%
    }%
  }%
}
\def\mint@sub#1#2_#3{%
  \@ifnextchar^{%
    \mint@sub@sup{#1}{#2}{#3}%
  }{%
    \mint@{#1}{#2}{#3}{}%
  }%
}
\def\mint@sup#1#2^#3{%
  \@ifnextchar_{%
    \mint@sup@sub{#1}{#2}{#3}%
  }{%
    \mint@{#1}{#2}{}{#3}%
  }%
}
\def\mint@sub@sup#1#2#3^#4{%
  \mint@{#1}{#2}{#3}{#4}%
}
\def\mint@sup@sub#1#2#3_#4{%
  \mint@{#1}{#2}{#4}{#3}%
}
\newcommand*{\mint@}[4]{%
  \mathop{}%
  \mkern-\thinmuskip
  \mathchoice{%
    \mint@@{#1}{#2}{#3}{#4}%
        \displaystyle\textstyle\scriptstyle
  }{%
    \mint@@{#1}{#2}{#3}{#4}%
        \textstyle\scriptstyle\scriptstyle
  }{%
    \mint@@{#1}{#2}{#3}{#4}%
        \scriptstyle\scriptscriptstyle\scriptscriptstyle
  }{%
    \mint@@{#1}{#2}{#3}{#4}%
        \scriptscriptstyle\scriptscriptstyle\scriptscriptstyle
  }%
  \mkern-\thinmuskip
  \int#1%
  \ifx\\#3\\\else_{#3}\fi
  \ifx\\#4\\\else^{#4}\fi
}
\newcommand*{\mint@@}[7]{%
  \begingroup
    \sbox0{$#5\int\m@th$}%
    \sbox2{$#5\int_{}\m@th$}%
    \dimen2=\wd0 %
    \let\mint@limits=#1\relax
    \ifx\mint@limits\relax
      \sbox4{$#5\int_{\kern1sp}^{\kern1sp}\m@th$}%
      \ifdim\wd4>\wd2 %
        \let\mint@limits=\nolimits
      \else
        \let\mint@limits=\limits
      \fi
    \fi
    \ifx\mint@limits\displaylimits
      \ifx#5\displaystyle
        \let\mint@limits=\limits
      \fi
    \fi
    \ifx\mint@limits\limits
      \sbox0{$#7#3\m@th$}%
      \sbox2{$#7#4\m@th$}%
      \ifdim\wd0>\dimen2 %
        \dimen2=\wd0 %
      \fi
      \ifdim\wd2>\dimen2 %
        \dimen2=\wd2 %
      \fi
    \fi
    \rlap{%
      $#5%
        \vcenter{%
          \hbox to\dimen2{%
            \hss
            $#6{#2}\m@th$%
            \hss
          }%
        }%
      $%
    }%
  \endgroup
}
\usepackage{mathrsfs}
\usepackage{amssymb}
\usepackage{amsmath}
\usepackage{amsthm}
\usepackage{amsfonts}
\usepackage{color}
\usepackage{graphicx}
\usepackage[active]{srcltx}
\usepackage{tikz}
\usepackage{pgflibraryarrows}
\usepackage{pgflibrarysnakes}

\usepackage[cp1252]{inputenc}

\usepackage{mathrsfs}
\usepackage{graphicx}

\usepackage[active]{srcltx}

\allowdisplaybreaks

\usepackage{titletoc}
\titlecontents{section}[0pt]{\addvspace{2pt}\filright}
{\contentspush{\thecontentslabel\ }}
{}{\titlerule*[8pt]{.}\contentspage}

\def\loc{{\mathop\mathrm{\,loc\,}}}

\def\bint{{\ifinner\rlap{\bf\kern.35em--}
\int\else\rlap{\bf\kern.45em--}\int\fi}\ignorespaces}

\def\bbint{{\ifinner\rlap{\bf\kern.35em--}
\hspace{0.078cm}\int\else\rlap{\bf\kern.45em--}\int\fi}\ignorespaces}

\newtheorem{thm}{Theorem}[section]
\newtheorem{lem}[thm]{Lemma}
\newtheorem{prop}[thm]{Proposition}

\numberwithin{equation}{section}
\newtheorem{conj}[thm]{Conjecture}

\title
{\Large\bf Uniqueness of diffeomorphic  minimizers of $L^p$-mean distortion
\footnotetext{\hspace{-0.35cm}
\endgraf
The author is supported by the Academy of Finland, project no. 334014.\\
Email: yizhu@jyu.fi\\
{\bf Key words phrases:} $L^p$-mean distortion, Hopf-Laplace equation, holomorphic quadratic differentials, uniqueness.\\
{\bf 2020 MSC:} 2020 Mathematics Subject Classification. Primary 31A05; Secondary 30C62, 35J25.
 \endgraf 
}}
\author{Yizhe Zhu}
 \date{}
\begin{document}

\arraycolsep=1pt
\allowdisplaybreaks

 \maketitle
 \begin{abstract}
We study the $\,L^p$-mean distortion functionals,
\[{\cal E}_p[f] = \int_\mathbb Y K^p_f(z) \; dz, \]
for Sobolev homeomorphisms $f \colon \overline{\mathbb Y}\xlongrightarrow{\rm onto}  \overline{\mathbb  X}$ where $\mathbb X$ and $\mathbb Y$ are  bounded simply connected Lipschitz domains, and $f$ coincides with a given boundary map $f_0 \colon \partial \mathbb Y \to \partial \mathbb X$. Here, $K_f(z)$ denotes the pointwise distortion function of $f$.  It is conjectured that for every $1 < p < \infty$, the functional $\mathcal{E}_p$ admits a minimizer that is a diffeomorphism. 
We prove that if such a diffeomorphic minimizer exists, then it is unique.
 \end{abstract}
\section{Introduction}
Let $\mathbb X$ and $\mathbb Y$ be  bounded simply connected Lipschitz domains in the complex plane $\mathbb C$ and $f \colon \mathbb Y \xlongrightarrow{\rm onto} \mathbb X$ a homeomorphism of Sobolev class $W^{1,1}_{\loc} (\mathbb{Y}, \mathbb{C})$. We say that 
$f$ has {\it finite distortion} if it satisfies
\begin{equation}\label{fd}
\|Df(z)\|^2 \leqslant 2 K (z) \, J(z,f) \qquad \text{ almost everywhere}
\end{equation}
for some measurable function $1 \leqslant K (z) < \infty$. Here, $\|Df(z)\|$ denotes  the Hilbert-Schmidt norm of the differential $Df$ and $J(z,f) = \det Df (z)$ is the Jacobian determinant of $f$ at $z$. The smallest such $K (z)$ will be denoted by $K_f(z)$ and is called the {\it distortion function} of $f$. Thus, under the assumption~\eqref{fd}, we have
\begin{equation}\label{distfn}
K_f(z) =  \frac{\|Df(z)\|^2}{2J(z,f)}, \; {\rm if }\; J(z,f)>0, \; {\rm and}\;\;  K_f (z) =
1, \; {\rm if}\;\; J(z,f)=0.
\end{equation}
We study homeomorphisms $f \colon \overline{\mathbb Y}\xlongrightarrow{\rm onto} \overline{\mathbb  X}$  of finite distortion with $K_f \in  L^p (\mathbb Y) $, $1 \leqslant p < \infty$. Denote
\begin{align*}
    \mathscr{E}_p[f]= \int_\mathbb Y K^p_f(z)\ dz,
\end{align*}

Let $ f_0 \colon \overline{\mathbb{Y}} \xrightarrow{\rm onto} \overline{\mathbb{X}}$ be a homeomorphism of finite distortion such that $ \mathscr{E}_p[f_0] < \infty$. We regard $ f_0 $ as the prescribed boundary data and define the corresponding classes of homeomorphic and diffeomorphic mappings as
\begin{equation}\label{homeomorphism,f_0,p}
\mathscr{H}^p_{f_0}(\overline{\mathbb{Y}}, \overline{\mathbb{X}}) := \left\{ f \in W^{1,1}_{\mathrm{loc}}(\mathbb{Y}, \mathbb{C}) : \mathscr{E}_p[f] < \infty,\ 
f|_{\partial \mathbb{Y}} = f_0|_{\partial \mathbb{Y}}, f:\overline{\mathbb{Y}}\xrightarrow{\rm onto}\overline{\mathbb{X}}\ {\rm is\ homeomorphism} \right\}
\end{equation}
and
\begin{equation}\label{diffeomorphism,f_0,p}
{\rm Diff}^p_{f_0}(\overline{\mathbb{Y}}, \overline{\mathbb{X}}):=\left\{f\in\mathscr{H}^p_{f_0}(\overline{\mathbb{Y}}, \overline{\mathbb{X}}): f:\mathbb{Y}\xrightarrow{\rm onto}\mathbb{X}\ {\rm is\ a\ }
C^\infty-{\rm diffeomorphism} \right\}.
\end{equation}
We recall the following conjecture from~\cite{iwaniec2014minimisers} and \cite{iwaniec2021energy}:
\begin{conj}\label{conj}
Let $\mathbb{D} \subset \mathbb{C}$ be the unit disk and $1<p<\infty$. In the space $\mathscr{H}^p_{f_0}(\overline{\mathbb{D}}, \overline{\mathbb{D}})$, there exists a minimizer $ f$ such that
\begin{equation}\label{con-D-p}
\mathscr{E}_p[f] = \min_{g \in \mathscr{H}^p_{f_0}(\overline{\mathbb{D}}, \overline{\mathbb{D}})} \mathscr{E}_p[g].
\tag{C-1}
\end{equation}

Moreover, this minimizer is a 
$C^{\infty}$-smooth diffeomorphism from $\mathbb{D}$ onto  $\mathbb{D}$.
\end{conj}
Even a stronger version of the conjecture remains open: for every \( p \in (1, \infty) \), it is conjectured that there is $f\in{\rm Diff}^p_{f_0}(\overline{\mathbb{Y}}, \overline{\mathbb{X}})$ such that
\begin{equation}\label{con-Y-p}
\mathscr{E}_p[f] = \min_{g \in {\rm Diff}^p_{f_0}(\overline{\mathbb{Y}}, \overline{\mathbb{X}})} \mathscr{E}_p[g].
\tag{C-2}
\end{equation}
Here, the class ${\rm Diff}^p_{f_0}(\overline{\mathbb{Y}}, \overline{\mathbb{X}})$ is assumed to be nonempty. At least, no counterexample is currently known.

We establish the uniqueness of such  minimizers, if they exists:
\begin{thm}\label{unique-diffeomorphism}
Let $\mathbb X$ and $\mathbb Y$ be  bounded simply connected domains in the complex plane $\mathbb C$, and $p\in (1, \infty)$. Suppose that the conjecture $\eqref{con-Y-p}$ holds,
then this minimizer is unique.
\end{thm}

Associated with the mean distortions are the polyconvex energy functionals for the inverse mapping $h = f^{-1} \colon  \mathbb X \xlongrightarrow{\rm onto} \mathbb Y$,~\cite{astala2005extremal, hencl2006regularity, hencl2007homeomorphisms,onninen2006regularity},
\begin{align}\label{p-energy}
\begin{aligned}
\mathscr E_p[f] &=  \frac{1}{2}\int_\mathbb X  K^{p-1}_h(x) \, ||Dh(x)||^2 \, dx,\qquad 1\le p<\infty \\
&=:\mathbb{E}_p[h]. 
\end{aligned}
\end{align}
Note that for $p=1$ and $n=2$ there is a fascinating connection, discovered in~\cite{astala2005extremal}, between minimizers of $L^1$-mean distortion and harmonic mappings. In particular, if $\mathbb{Y}$ is convex, given a homeomorphism $f_0:\overline{\mathbb{X}}\xlongrightarrow{\rm onto}\overline{\mathbb{Y}}$ with $\mathscr{E}_{p}[f_0]<\infty$, the Rad\'o-Kneser-Choquet theorem \cite{duren2004harmonic} asserts that there exists a $C^\infty$-diffeomorphism $g$ which solves the minimization problem \eqref{con-D-p}, and the invese $f$ is harmonic in $\mathbb{X}$. If the target $\mathbb{Y}$ is not convex there always exists a boundary homeomorphism $h_0: \partial\mathbb{X}\xrightarrow{\rm onto}\partial\mathbb{Y}$ whose harmonic extension takes points in
$\mathbb{X}$ beyond $\overline{\mathbb{Y}}$, that is, the minimization problem \ref{con-Y-p} fails in general if $p=1$. This was already observed by Choquet \cite{choquet1945type}, see also \cite{kalaj2010invertible}. 

The outer variation of a given energy
integral leads to the Euler-Lagrange equation. This equation is not available when the energy integral is restricted to homeomorphisms. When dealing with Sobolev homeomorphisms, one should perform the inner variation of given integral (see \cite{bauman1991maximal,cristina2014hopf,iwaniec2013lipschitz,sandier2003limiting}, etc.,  simply a change of variable in the interpreted variables. 

The inner variation of the energy $\mathbb{E}_p[h]$ leads to the equation

\begin{align}\label{p-hopf}
\frac{\partial}{\partial \bar z}\left(K_h^{p-1}h_z\overline{h_{\bar z}}\right)=0.
\end{align}
This suggests us to study an auxiliary minimization problem. Given a positive and continuous function $\Phi: \overline{\mathbb{Y}}\to (0,\infty]$, we define the {\em $\Phi$-weighted Dirichlet energy}
\begin{align}\label{phi-energy}
\mathscr{E}^{\Phi}_{\mathbb{X}}[h]:=\int_{\mathbb{X}}\Phi(h)|Dh|^2dz
\end{align}
subjected to homeomorphisms $h:\overline{\mathbb{X}}\xlongrightarrow{\rm onto}\overline{\mathbb{Y}}$ with $h\in W^{1,2}(\mathbb{X},\mathbb{R}^2)$. This gives rise to the so-called {\em $\Phi$-Hopf-Laplace
equation},
\begin{align}\label{phi-laplacian}
	\frac{\partial}{\partial\bar z}\Big(\Phi(h(z))h_z\overline{h_{\bar z}}\Big)=0,
\end{align}
for $h\in W^{1,2}(\mathbb{X},\mathbb{Y})$ with finite $\Phi$-weighted Dirichlet energy. 

While completing this manuscript, we became aware of the recent work by Martin and Yao~\cite{MartinYao}, where related uniqueness results for extremal mappings of finite distortion are established. Our approach, developed independently, is new and, hopefully,  of independet interest. In fact, we prove that  any diffeomorphic solution is
\[
h \colon \overline{\mathbb{X}} \xrightarrow{\rm onto} \overline{\mathbb{Y}}, \quad h = f_0^{-1} \text{ on } \partial \mathbb{X},
\]
to the inner-variational equation \eqref{p-hopf} is the unique minimizer of the $\Phi$-weighted Dirichlet energy functional, where $\Phi(z) = K_{h^{-1}}^{p-1}(z)$. This uniqueness result for the weighted energy implies, in particular, that $h$ is also the unique solution to the inner-variational equation. We hope that this new perspective contributes further insight toward Conjecture \ref{conj}.\\

\noindent{\bf Acknowledgements.} The author thanks his thesis advisor Jani Onninen for his ideas and all his help in improving the manuscript.

\section{Prerequisites}
In this section we review from \cite{strebel1984quadratic} useful concepts and results about
holomorphic quadratic differentials $\varphi(z)dz\otimes dz$ and their trajectories. 
\subsection{An Integral Identity}
First, we show a powerful identity that will play an important role in our proof. The idea of the proof is from \cite[Lemma 8.1]{iwaniec2013mappings}.
\begin{lem}\label{key-lemma}
Let $\mathbb{X}$ and $ \mathbb{Y}$ be bounded domains in $\mathbb{C}$. Suppose that $h: \mathbb{X}\xrightarrow{\rm onto}\mathbb{Y}$
and $H:\mathbb{X}\xrightarrow{\rm onto}\mathbb{Y}$ are orientation preserving $C^\infty$-diffeomorphisms of finite $\Phi$-weighted Dirichlet energy. Define $f:= H^{-1}\circ h: \mathbb{X}\xrightarrow{\rm onto}\mathbb{X}$. Then we have
\begin{equation}\label{Integral-Identity}
\begin{aligned}
	\mathscr{E}^{\Phi}_{\mathbb{X}}[H]-\mathscr{E}^{\Phi}_{\mathbb{X}}[h]
	&=4\int_{\mathbb{X}}
	\left[\frac{|f_z-\gamma(z)f_{\bar z}|^2}{|f_z|^2-|f_{\bar z}|^2}-1\right]\Phi(h)|h_zh_{\bar z}|dz\\
&\quad +4\int_{\mathbb{X}}\Phi(h)\cdot\frac{(|h_z|-|h_{\bar z}|)^2|f_{\bar z}|^2}{|f_z|^2-|f_{\bar z}|^2}dz,
\end{aligned}
\end{equation}
where
\begin{equation*}
\gamma(z)=
\left\{
\begin{aligned}
	&h_z\overline{h_{\bar z}}|h_z\overline{h_{\bar z}}|^{-1}\quad {\rm if}\ h_zh_{\bar z}\neq 0\\
	&0 \qquad\qquad\qquad{\rm otherwise}.
\end{aligned}
\right.
\end{equation*}
The integrals in \eqref{Integral-Identity} converge.
\end{lem}
\begin{proof}
It is worth noting that $f:\mathbb{X}\xrightarrow{\rm onto}\mathbb{X}$ need not have finite energy. The convergence of the integrals, not obvious at first glance, is a consequence of the finite $\Phi$-weighted Dirichlet energy assumption imposed on the mappings $h$ and $H$.

We begin with the chain rule applied to 
$H=h\circ f^{-1}: \mathbb{X}\xlongrightarrow{\rm onto}\mathbb{Y}$,
\begin{align*}
	\frac{\partial H(\omega)}{\partial\omega}=h_z(z)\frac{\partial f^{-1}}{\partial\omega}+h_{\bar z}(z)\frac{\overline{\partial f^{-1}}}{\partial\overline\omega}\\
	\frac{\partial H(\omega)}{\partial\overline\omega}=h_z(z)\frac{\partial f^{-1}}{\partial\overline\omega}+h_{\bar z}(z)\frac{\overline{\partial f^{-1}}}{\partial\omega},	
\end{align*}
where $\omega=f(z)$. We express the complex partial derivatives of $f^{-1}: \mathbb{X}\to\mathbb{X}$ at
$\omega$ in terms $f_z(z)$ and $f_{\bar z}(z)$ at $z=f^{-1}(\omega)$,
\begin{align*}
	\frac{\partial f^{-1}}{\partial\omega}
	=\frac{\overline{f_z(z)}}{J(z,f)}\quad {\rm and}\quad
	\frac{\partial f^{-1}}{\partial\overline\omega}
	=-\frac{f_{\bar z}(z)}{J(z,f)}.
\end{align*}
Note that the Jacobian determinant $J(z,f)=|f_z|^2-|f_{\bar z}|^2$ is strictly positive. These expressions yield
\begin{align*}
	\frac{\partial H}{\partial\omega}=\frac{h_z\overline{f_z}-h_{\bar z}\overline{f_{\bar z}}}{|f_z|^2-|f_{\bar z}|^2}\quad{\rm and}\quad
	\frac{\partial H}{\partial\overline\omega}
	=\frac{h_{\bar z}{f_z}-h_{ z}f_{\bar z}}{|f_z|^2-|f_{\bar z}|^2}
\end{align*}
Next, note that $J_H=J_h/J_f$ and we compute the energy of $H$ over the set $f(\mathbb{X})=\mathbb{X}$ by substituting $\omega=f(z)$,
\begin{equation}\label{Energy-H}
	\begin{aligned}
		\mathscr{E}^{\Phi}_{f(\mathbb{X})}[H]
		&=2\int_{f(\mathbb{X})}\Phi(H(\omega))\Big(|H_\omega|^2+|H_{\bar\omega}|^2\Big)d\omega\\
		&=2\int_{\mathbb{X}}\Phi(h)\frac{|h_z\overline{f_z}-h_{\bar z}\overline{f_{\bar z}}|^2+|h_{\bar z}{f_z}-h_{ z}f_{\bar z}|^2}{|f_z|^2-|f_{\bar z}|^2}dz
	\end{aligned}
\end{equation}
On the other hand, the energy of $h$ over the set $\mathbb{X}$ equals
\begin{equation}\label{Energy-h}
	\begin{aligned}
		\mathscr{E}^{\Phi}_{\mathbb{X}}[h]
		&=2\int_{\mathbb{X}}\Phi(h)
		\Big(|h_z|^2+|h_{\bar z}|^2\Big)dz.
	\end{aligned}
\end{equation}
The desired formula follows by subtracting these two integrals,
\begin{equation}\label{minus}
\begin{aligned}
\mathscr{E}^{\Phi}_{\mathbb{X}}[H]-\mathscr{E}^{\Phi}_{\mathbb{X}}[h]
&=4\int_{\mathbb{X}}\Phi(h)\cdot\frac{\Big(|h_z|^2+|h_{\bar z}|^2\Big)\cdot|f_{\bar z}|^2-2{\rm Re}[h_z\overline{h_{\bar z}}\overline{f_z}f_{\bar z}]}
{|f_z|^2-|f_{\bar z}|^2}dz\\
&=4\int_{\mathbb{X}}\Phi(h)\cdot\frac{2|h_zh_{\bar z}||f_{\bar z}|^2-2{\rm Re}[h_z\overline{h_{\bar z}}\overline{f_z}f_{\bar z}]}
{|f_z|^2-|f_{\bar z}|^2}dz\\
&\quad+4\int_{\mathbb{X}}\Phi(h)\cdot\frac{(|h_z|-|h_{\bar z}|)^2|f_{\bar z}|^2}{|f_z|^2-|f_{\bar z}|^2}dz\\
&=4\int_{\mathbb{X}}
\left[\frac{|f_z-\gamma(z)f_{\bar z}|^2}{|f_z|^2-|f_{\bar z}|^2}-1\right]\Phi(h)|h_zh_{\bar z}|dz\\
&\quad+4\int_{\mathbb{X}}\Phi(h)\cdot\frac{(|h_z|-|h_{\bar z}|)^2|f_{\bar z}|^2}{|f_z|^2-|f_{\bar z}|^2}dz
\end{aligned}
\end{equation}
\end{proof}
\subsection{Holomorphic quadratic differentials}
Let $\varphi(z)dz\otimes dz$ be a
holomorphic quadratic differential in $\mathbb{X}$ with isolated zeros, called {\em critical points}. Through every noncritical point there pass two $C^\infty$-smooth
orthogonal arcs. A {\em vertical arc} is a $C^\infty$-smooth curve $\gamma=\gamma(t)$, $a<t<b$, along which
\begin{align*}
	[\gamma'(t)]^2\phi(\gamma(t))<0,\qquad a<t<b.
\end{align*}
A {\em vertical trajectory} of $\phi$ in $\mathbb{X}$ is a maximal vertical arc, that is, not properly contained in any other vertical arc. The {\em horizontal arcs} and
{\em horizontal trajectories} are defined in an exactly similar way, via the opposite inequality. Through every noncritical point of $\phi$ there passes a unique vertical (horizontal) trajectory. A trajectory whose closure contains a critical point of $\phi$ is called a {\em critical trajectory}. There are at most a countable number of critical trajectories.

Every noncritical vertical trajectory $\gamma\subset\mathbb{U}$ in a simply connected
domain $\mathbb{U}$ is a {\em cross cut}, see  \cite[Theorem 15.1]{strebel1984quadratic}.
\begin{lem}\label{minimum-vertical-arc}
Consider a vertical arc $\gamma\subset U$ in a simply connected domain $U$. Let $\beta$ be any locally rectifiable curve in $U$ which contains the endpoints of $\gamma$. Then
\begin{align}
	\int_\gamma|\varphi|^{1/2}|dz|\le
	\int_\beta|\varphi|^{1/2}|dz|.
\end{align}
\end{lem}

For the proof of this lemma we refer to \cite[Theorem 16.1]{strebel1984quadratic}.
\begin{lem}[Fubini-like integration formula]
Let $\varphi(z)dz\otimes dz$ be a
holomorphic quadratic differential in a simply connected domain $U$, $\varphi\not\equiv0$. Suppose that $F$ and $G$ are measurable functions in $\mathbb{U}$ such that
\begin{align}
	\int_{U}|\varphi(z)||F(z)|dz<\infty\quad{\rm and} \int_{U}|\varphi(z)||G(z)|dz<\infty.
\end{align}
Then for almost every vertical trajectory $\gamma$ of $\varphi(z)dz\otimes dz$, we have
\begin{align}
	\int_{\gamma}|\varphi(z)|^{1/2}|F(z)||dz|<\infty\quad{\rm and} \int_{\gamma}|\varphi(z)|^{1/2}|G(z)||dz|<\infty.
\end{align}

\quad$\bullet$ If 
\begin{align}
	\int_{\gamma}|\varphi(z)|^{1/2}F(z)|dz|
	=\int_{\gamma}|\varphi(z)|^{1/2}G(z)|dz|,
\end{align}
for almost every vertical trajectory $\gamma$ of $\varphi(z)dz\otimes dz$ then
\begin{align}
	\int_{\mathbb{U}}|\varphi(z)|F(z)dz
	=\int_{\mathbb{U}}|\varphi(z)|G(z)dz.
\end{align}

\quad$\bullet$ If 
\begin{align}\label{Fubini's 1}
	\int_{\gamma}|\varphi(z)|^{1/2}F(z)|dz|
	\le\int_{\gamma}|\varphi(z)|^{1/2}G(z)|dz|,
\end{align}
for almost every vertical trajectory $\gamma$ of $\varphi(z)dz\otimes dz$ then
\begin{align}\label{Fubini's 2}
	\int_{\mathbb{U}}|\varphi(z)|F(z)dz
	\le\int_{\mathbb{U}}|\varphi(z)|G(z)dz.
\end{align}
\end{lem}

Again, for the proof of this lemma, we refer to \cite{strebel1984quadratic}.

Given a quadratic holomorphic differential $\varphi(z) dz\otimes dz$ we define two partial differential operators, called the {\em horizontal} and {\em vertical derivatives}
\begin{align*}
	\partial_{\rm H}=\frac{\partial }{\partial z}+\frac{\varphi}{|\varphi|}\frac{\partial}{\partial\bar z}\quad{\rm and}\quad 
	\partial_{\rm v}=\frac{\partial }{\partial z}-\frac{\varphi}{|\varphi|}\frac{\partial}{\partial\bar z}.
\end{align*}
If $h$ satisfies the $\Phi$-Hopf-Laplace equation $\Phi(h)h_z\overline{h_{\bar z}}=\varphi$, then the horizontal and vertical trajectories of $\varphi(z)dz\otimes dz$ are the lines of maximal and minimal stretch for $h$. Precisely, the following identities hold:
\begin{equation}\label{partial-Jh-1}
	\begin{aligned}
	&|\partial_{\rm H}h|=|h_z|+|h_{\bar z}|,\quad |\partial_{\rm v}h|=|h_z|-|h_{\bar z}|,\\
	&|\partial_{\rm H}h|\cdot|\partial_{\rm v}h|=|J_h|,\quad
	\Phi(h)\left(|\partial_{\rm H}h|^2-|\partial_{\rm v}h|^2\right)=4|\varphi|.
	\end{aligned}
\end{equation}
Here and after $J_h={\rm det} Dh$. As a consequence
\begin{align}\label{partial-Jh-2}
	|\partial_{\rm v}h|^2\le J_h\le|\partial_{\rm H}h|^2.
\end{align}
\section{Uniqueness of weighted Dirichlet energy minimizers}

In this section, we prove that the minimizer of the $\Phi$-weighted energy \eqref{phi-energy} is unique. Let $ g_0 \colon \overline{\mathbb{X}} \xrightarrow{\rm onto} \overline{\mathbb{Y}}$ be a homeomorphism, diffeomorphic in $\mathbb{X}$, of finite distortion such that $ \mathscr{E}^\Phi_\mathbb{X}[g_0] < \infty$. In a manner similar to \eqref{homeomorphism,f_0,p} and \eqref{diffeomorphism,f_0,p}, we regard $g_0$ as the prescribed boundary data and define the corresponding classes of homeomorphic and diffeomorphic mappings as
\begin{equation}\label{homeomorphism,g_0,Phi}
\mathscr{H}^\Phi_{g_0}(\overline{\mathbb{X}}, \overline{\mathbb{Y}}) := \left\{ g \in W^{1,1}_{\mathrm{loc}}(\mathbb{X}, \mathbb{C}) : \mathscr{E}^\Phi_{\mathbb{X}}[g] < \infty,\ 
g|_{\partial \mathbb{X}} = g_0|_{\partial \mathbb{X}},\ g:\overline{\mathbb{X}}\xrightarrow{\rm onto}\overline{\mathbb{Y}}\ {\rm is\ homeomorphic} \right\}
\end{equation}
and
\begin{equation}\label{diffeomorphism,g_0,Phi}
{\rm Diff}^\Phi_{g_0}(\overline{\mathbb{X}}, \overline{\mathbb{Y}}):=\left\{g\in\mathscr{H}^p_{g_0}(\overline{\mathbb{X}}, \overline{\mathbb{Y}}): g:\mathbb{X}\xrightarrow{\rm onto}\mathbb{Y}\ {\rm is\ a\ }
C^\infty-{\rm diffeomorphism} \right\}.
\end{equation}
Clearly, $\mathscr{H}^\Phi_{g_0}(\overline{\mathbb{X}}, \overline{\mathbb{Y}})$
and ${\rm Diff}^\Phi_{g_0}(\overline{\mathbb{X}}, \overline{\mathbb{Y}})$ are not empty.

To prove the uniqueness of the minimizer of the 
$\Phi$-weighted Dirichelt energy, we make the conjecture as follows:
Given a continuous function $\Phi: \mathbb{Y}\to (0,\infty]$, there exists some $h\in{\rm Diff}^\Phi_{g_0}(\overline{\mathbb{X}}, \overline{\mathbb{Y}})$ such that 
\begin{align}\label{exsitence assumpiton}
\int_{\mathbb{X}}\Phi(h)|Dh|^2dz
=\inf_{g\in{\rm Diff}^\Phi_{g_0}(\overline{\mathbb{X}},\overline{\mathbb{Y}})}\int_{\mathbb{X}}\Phi(g)|Dg|^2dz.
\tag{$C$-$\Phi$}
\end{align}
Then we establish the uniqueness of such  minimizers, if they exists:
\begin{prop}\label{Uniqueness-Phi}
Let $\mathbb X$ and $\mathbb Y$ be  bounded simply connected domains in the complex plane $\mathbb C$. Suppose that the conjecture $\eqref{exsitence assumpiton}$ holds,
then this minimizer is unique.
\end{prop}

Before proceeding the uniqueness of $\Phi$-Hopf-harmonic diffeomorphisms, let us give an equivalent characterization. 
\subsection{  $\Phi$-Hopf-harmonic diffeomorphisms are the energy minimizers}\label{hopf-harmonic}
It's known that, a diffeomorphism $h:\mathbb{X}\xrightarrow{\rm onto}\mathbb{Y}$ which minimizes the $\Phi$-weighted Dirichlet energy in ${\rm Diff}^\Phi_{g_0}(\overline{\mathbb{X}}, \overline{\mathbb{Y}})$ is a $\Phi$-Hopf-harmonic diffeomorphism; see\cite{iwaniec2013lipschitz}. Actually, the converse also holds.
\begin{prop}\label{phi-laplacian-minimum}
Let $\mathbb{X}$ and $\mathbb{Y}$ be bounded simply connected domains in the complex plane $\mathbb{C}$. Let $g_0: \overline{\mathbb{X}}\xlongrightarrow{\rm onto}\overline{\mathbb{Y}} $ be an homeomorphism, diffeomorphic in $\mathbb{X}$, of Sobolev class $W^{1,2}(\mathbb{X},\mathbb{C})$ with finite $\Phi$-weighter energy. Suppose that $h\in{\rm Diff}^\Phi_{g_0}(\overline{\mathbb{X}},  \overline{\mathbb{Y}})$. Then $h$ is $\Phi$-Hopf-harmonic if and only if
\begin{align*}
\int_{\mathbb{X}}\Phi(h)|Dh|^2
=\inf_{g\in { \rm Diff}^\Phi_{g_0}(\overline{\mathbb{X}},\overline{\mathbb{Y}})}\int_{\mathbb{X}}\Phi(g)|Dg|^2.
\end{align*}
\end{prop}
\begin{proof}
	Let $h\in { \rm Diff}^\Phi_{g_0}(\overline{\mathbb{X}},\overline{\mathbb{Y}})$ be a $\Phi$-Hopf-harmonic diffeomorphism. Then
\begin{align*}
	\Phi(h)h_z\overline{h_{\bar z}}=\varphi
	\quad {\rm for \ some\ holomorphic\ }\varphi\not\equiv 0.
\end{align*}
Let $H\in { \rm Diff}^\Phi_{g_0}(\overline{\mathbb{X}},\overline{\mathbb{Y}})$. Define
\begin{align*}
	f=H^{-1}\circ h:\mathbb{X}\xrightarrow{\rm onto}\mathbb{X}.
\end{align*}
In view of Lemma \ref{key-lemma}, we see that
\begin{equation}\label{approx-1}
	\begin{aligned}
		\mathscr{E}^{\Phi}_{\mathbb{X}}[H]-\mathscr{E}^{\Phi}_{\mathbb{X}}[h]
		&=4\int_{\mathbb{G}}
		\left[\frac{|f_z-\frac{\varphi}{|\varphi^j|}f_{\bar z}|^2}{|f_z|^2-|f_{\bar z}|^2}-1\right]|\varphi|dz\\
		&\quad +4\int_{\mathbb{X}}\Phi(h)\cdot\frac{(|h_z|-|h_{\bar z}|)^2|f_{\bar z}|^2}{|f_z|^2-|f_{\bar z}|^2}dz.
	\end{aligned}
\end{equation}
Since $f$ is a sense-preserving diffeomorphism and $\Phi$ is positive, the last integral in \eqref{approx-1} is nonnegative,
\begin{align*}
\int_{\mathbb{X}}\Phi(h^j)\cdot\frac{(|h_z|-|h_{\bar z}|)^2|f_{\bar z}|^2}{|f_z|^2-|f_{\bar z}|^2}dz\ge 0.
\end{align*}
We estimate the first integral in 
\eqref{approx-1} by H\"older's inequality,
\begin{equation}\label{first-part}
	\int_{\mathbb{X}}
	\frac{|f_z-\frac{\varphi}{|\varphi|}f_{\bar z}|^2}{|f_z|^2-|f_{\bar z}|^2}|\varphi|dz
\ge\frac{\left(\int_{\mathbb{X}}|f^j_z-\frac{\varphi}{|\varphi|}f_{\bar z}|\sqrt{|\varphi|}\sqrt{|\varphi\circ f|}dz\right)^2}
{\int_{\mathbb{X}}J_{f}|\varphi\circ f|dz}.
\end{equation}
The denominator is bounded from above, by the $\mathscr{L}^1$-norm of $\varphi$,
\begin{align*}
	\int_{\mathbb{X}}J_{f}|\varphi\circ f|dz
=\int_{\mathbb{X}}|\varphi|dz.
\end{align*}
By \eqref{approx-1} and \eqref{first-part}, we have
\begin{equation}\label{approx-3}
\begin{aligned}
\mathscr{E}^{\Phi}_{\mathbb{X}}[H]-\mathscr{E}^{\Phi}_{\mathbb{X}}[h]
&\ge4\frac{\left(\int_{\mathbb{G}}|f_z-\frac{\varphi}{|\varphi|}f_{\bar z}|\sqrt{|\varphi|}\sqrt{|\varphi\circ f|}dz\right)^2}
{\int_{\mathbb{X}}|\varphi|dz}
-4\int_{\mathbb{X}}|\varphi|\\
&=4\frac{\left(\int_{\mathbb{G}}|\partial_{\rm v}f|\sqrt{|\varphi|}\sqrt{|\varphi\circ f|}dz\right)^2}
{\int_{\mathbb{X}}|\varphi|dz}
-4\int_{\mathbb{X}}|\varphi|
\end{aligned}
\end{equation}
Concerning the numerator, we shall make use of Fubini's theorem. First, we change the variables in line integrals over thevertical trajectories. Namely, for almost every vertical noncritical trajectory $\gamma$ it
holds that
\begin{align}\label{approx-3.5}
\int_{\gamma}|\partial_vf|\sqrt{|\varphi\circ f|}=\int_{f(\gamma)}\sqrt{|\varphi|}.
\end{align}
Since $f$ is the identity map on $\partial\mathbb{X}$ and the trajectory has two distinct endpoints $x_1,x_2$ on $\partial\mathbb{X}$.
Therefore, $f(\gamma)$  has
the same endpoints as $\gamma$. Now, by Lemma \ref{minimum-vertical-arc} we have 
\begin{align*}
	\int_{f(\gamma|)}\sqrt{|\varphi|}|dz|
	\ge\int_{\gamma}\sqrt{|\varphi|}|dz|.
\end{align*}
Therefore,
\begin{align*}
\int_{\gamma|}|\partial_{\rm v}f|\sqrt{|\varphi\circ f|}\ge\int_{\gamma}\sqrt{|\varphi|}.
\end{align*}
Fubini's integration formula \eqref{Fubini's 2} yields
\begin{align}\label{approx-4}
\int_{\mathbb{X}}|\partial_{\rm v}f|\sqrt{|\varphi\circ f|}\sqrt{|\varphi|}\ge\int_{\mathbb{X}}|\varphi|.
\end{align}
Combining \eqref{approx-3} and \eqref{approx-4}, we obtain
\begin{align*}
	\mathscr{E}^{\Phi}_{\mathbb{X}}[H]-\mathscr{E}^{\Phi}_{\mathbb{X}}[h]\ge 0.
\end{align*}
This also finishes the proof of Proposition \ref{phi-laplacian-minimum}.
\end{proof}
\subsection{Uniqueness, proof of Proposition \ref{Uniqueness-Phi}}\label{uniqueness-phi}
First, under the assumption \eqref{exsitence assumpiton}, suppose that $h\in {\rm Diff}^\Phi_{g_0}(\overline{\mathbb{X}},\overline{\mathbb{Y}})$ is a minimizer of the $\Phi$-weighter Dirichlet energy subject to Sobolev diffeomorphic mappings in ${\rm Diff}^\Phi_{g_0}(\overline{\mathbb{X}},\overline{\mathbb{Y}})$. Let  ${\rm Diff}^\Phi_{g_0}(\overline{\mathbb{X}},\overline{\mathbb{Y}})$ be a $\Phi$-Hopf-harmonic diffeomorphism, we will prove that
\begin{align*}
    h(z)=H(z)\quad {\rm for\ all\ } z\in\mathbb{X}.
\end{align*}
By Proposition \ref{phi-laplacian-minimum},
\begin{align*}
	\Phi(h)h_z\overline{h_{\bar z}}&=\varphi\qquad {\rm for\ some \ holomorphic\ }\varphi\\
	\Phi(H)H_z\overline{H_{\bar z}}&=\psi\qquad {\rm for\ some \ holomorphic\ }\psi,
\end{align*}
and $H$ minimizes the $\Phi$-weighted Dirichlet energy subject to Sobolev diffeomorphisms in ${\rm Diff}^\Phi_{g_0}(\overline{\mathbb{X}},\overline{\mathbb{Y}})$. 
Fix a disk $\mathbb{D}\Subset\mathbb{X}$ and set $\mathbb{F}:=h^{-1}\circ H(\overline{D})$, then $\mathbb{F}\subset\mathbb{X}$ is compact.
In view of \eqref{Integral-Identity}, it follows that
\begin{equation}\label{unique-part-1}
\begin{aligned}
    \mathscr{E}^{\Phi}_{\mathbb{X}}[H]-\mathscr{E}^{\Phi}_{\mathbb{X}}[h]
&=4\int_{\mathbb{X}}
\frac{|\partial_{v}f|^2}{J_f}|\varphi|dz-4\int_{\mathbb{X}}|\varphi|dz\\
&\quad+4\int_{\mathbb{X}}\Phi(h)\cdot\frac{(|h_z|-|h_{\bar z}|)^2|f_{\bar z}|^2}{J_f}dz.
\end{aligned}
\end{equation}
We estimate the first integral in 
\eqref{unique-part-1} by H\"older's inequality,
\begin{align*}
\int_{\mathbb{X}}
\frac{|\partial_vf|^2}{J_f}|\varphi|dz
\ge\frac{\left(\int_{\mathbb{X}}|\partial_vf|\sqrt{|\varphi|}\sqrt{|\varphi\circ f|}dz\right)^2}
{\int_{\mathbb{X}}J_{f}|\varphi\circ f|dz}.
\end{align*}
Following the arguments in Proposition \ref{phi-laplacian-minimum}, we obtain
\begin{align}\label{unique-part-2}
\int_{\mathbb{X}}\frac{|\partial_v f|^2}{J_f}|\varphi|dz\ge \int_\mathbb{X}|\varphi|dz.
\end{align}
This together with \eqref{unique-part-1} gives
\begin{align*}
\mathscr{E}^{\Phi}_{\mathbb{X}}[H]-\mathscr{E}^{\Phi}_{\mathbb{X}}[h]
&\ge4\int_{\mathbb{X}}\Phi(h)\cdot\frac{(|h_z|-|h_{\bar z}|)^2|f_{\bar z}|^2}{J_f}dz\\
&\ge4\int_{\mathbb{F}}\Phi(h)\cdot\frac{(|h_z|-|h_{\bar z}|)^2|f_{\bar z}|^2}{J_f}dz.
\end{align*}
Since $h$ is an orientation-preserving diffeomorphism on $\mathbb{X}$, and $\Phi(h)$ is continuous on $\mathbb{X}$, we have
$|h_z|-|h_{\bar z}|,\Phi(h)\ge c>0$ for every $z\in\mathbb{F}\Subset\mathbb{X}$ and constants $c(\mathbb{F})>0$
\begin{equation}\label{step3-uniquness}
\begin{aligned}
\mathscr{E}^{\Phi}_{\mathbb{X}}[H]-\mathscr{E}^{\Phi}_{\mathbb{X}}[h]
&\ge4c^3\int_{\mathbb{F}}\frac{|f_{\bar z}|^2}{J_{f}}dz=4c^3\int_{f(\mathbb{F})}|g_{\bar \omega}|^2d\omega\\
&\ge 4c^3\int_{\mathbb{D}}|g_{\bar \omega}|^2d\omega.
\end{aligned}
\end{equation}
Here we have made the substitution $z=g^k(\omega)$. We see that $g_{\bar\omega}=0$ on $\mathbb{D}$. But $\mathbb{D}\subset\mathbb{X}$ is arbitrary, so $f: \mathbb{X}\xlongrightarrow{\rm onto} \mathbb{X}$ and $g=f^{-1}: \mathbb{X}\xlongrightarrow{\rm onto} \mathbb{X}$ are conformal.

Next, we are going to show that $f(z)=z$. Note that $\overline{h^{-1}(\mathbb{Y})}=
\partial\mathbb{X}$. Now, the conformal mapping $f:\mathbb{X}\xrightarrow{\rm onto}\mathbb{X}$ extends continuously to $\partial\mathbb{X}$. Since $h(z)=H(z)$ on $\partial\mathbb{X}$, we have $f(z)=z$ on $\partial\mathbb{X}$. Finally, we appeal
to the general fact that two holomorphic functions in $\mathbb{X}$, continuous on $\overline{\mathbb{X}}$, are the same if they coincide on an arc of $\overline{\mathbb{X}}$. Therefore, $f(z)=z$ in $\mathbb{X}$, which means that $h(z)=H(z)$ for all $z\in\mathbb{X}$.
\section{Proof of Theorem \ref{unique-diffeomorphism}}
In this section, we will prove the uniqueness of the $C^\infty$-diffeomorphic minimizer under the conjecture \eqref{con-Y-p}.

Given the boundary data $f_0$, we define the subclass of $\mathscr{H}^p_{f_0}(\overline{\mathbb{Y}},\overline{\mathbb{X}})$ (see \eqref{homeomorphism,f_0,p}) consisting of diffeomorphic mappings
\begin{equation}\label{definition-Diff_f0}
    \text{Diff}^p_{f_0}(\overline{\mathbb{Y}},\overline{\mathbb{X}}):=\left\{ f\in\mathscr{H}^p_{f_0}(\overline{\mathbb{Y}},\overline{\mathbb{X}}): f:\mathbb{Y}\xrightarrow{\rm onto}\mathbb{X}\ {\rm is\ diffeomorphic} \right\}.
\end{equation}

\begin{lem}\label{basic inequality phi-p}
Let $f$ be a $C^\infty$-diffeomorphic minimizer satisfying \eqref{con-Y-p} and  $\Phi(\cdot):=K_{f}^{p-1}(\cdot)$. For any $g$ with $g^{-1}\in{\rm Diff}^p_{f_0}(\overline{\mathbb{Y}},\overline{\mathbb{X}})$, we obtain
\begin{align}\label{comparison-two-energy}
\mathscr{E}^{\Phi}_{\mathbb{X}}[g]\le \mathscr{E}_p[f]^{\frac{p-1}{p}}\cdot\mathscr{E}_p[g^{-1}]^{\frac{1}{p}}
\end{align}
\end{lem}
\begin{proof}
By the definition of $\mathscr{E}^{\Phi}_{\mathbb{X}}[\cdot]$ and H\"older's inequality,
\begin{align*}
\mathscr{E}^{\Phi}_{\mathbb{X}}[g]
&=\int_{\mathbb{X}}K^{p-1}_{f}(g)J_g^{\frac{p-1}{p}}\cdot\frac{|Dg|^2}{J_g^{\frac{p-1}{p}}}dz\\
&\le\left(\int_{\mathbb{X}}K_{f}^p(g)J_gdz\right)^{\frac{p-1}{p}}\cdot\left(\int_{\mathbb{X}}\frac{|Dg|^{2p}}{J_g^{p-1}}dz\right)^{\frac{1}{p}}\\
&=\left(\int_{\mathbb{Y}}K_{f}^p(y)dy\right)^{\frac{p-1}{p}}\cdot\left(\int_{\mathbb{X}}K_g^{p-1}(z)|Dg|^2dz\right)^{\frac{1}{p}}\\
&=\mathscr{E}_p[f]^{\frac{p-1}{p}}\cdot\mathscr{E}_p[g^{-1}]^{\frac{1}{p}}
\end{align*}
\end{proof}
\begin{lem}\label{minimizer}
Let $f$ be a $C^\infty$-diffeomorphic minimizer satisfying \eqref{con-Y-p} and  $\Phi(\cdot):=K_{f}^{p-1}(\cdot)$. For any $g^{-1}\in{\rm Diff}^p_{f_0}(\overline{\mathbb{Y}},\overline{\mathbb{X}})$, define the $\Phi$-weighted Dirichlet energy as
\begin{align*}
	\mathscr{E}^{\Phi}_{\mathbb{X}}[g]:=\int_{\mathbb{X}}K_{f}^{p-1}(g)|Dg|^2dz.
\end{align*}
Then $h:=f^{-1}$ is the unique minimizer of the auxiliary $\mathscr{E}^{\Phi}_{\mathbb{X}}[\cdot]$ subject to the class of ${\rm Diff}^p_{f_0}(\overline{\mathbb{Y}},\overline{\mathbb{X}})$. Futhermore, 
	\begin{align}\label{phi=p}
		\mathscr{E}^{\Phi}_{\mathbb{X}}[h]=\mathscr{E}_p[f].
	\end{align}
\end{lem}
\begin{proof}
Since $f$ is a $C^\infty$-diffeomorphic minimizer satisfying \eqref{con-Y-p}, a change of variables shows that the inverse map $h:=f^{-1}$ minimizes the energy $\mathbb{E}_p[\cdot]$, (see \eqref{p-energy}), among the Sobolev diffeomorphisms whose inverse belong to ${\rm Diff}^p_{f_0}(\overline{\mathbb{Y}},\overline{\mathbb{X}})$. Hence, we have 
\begin{align*}
    \frac{\partial}{\partial\bar z}\left(K^{p-1}_h h_z\overline{h_{\bar z}}\right)= \frac{\partial}{\partial\bar z}\left(K^{p-1}_f(h) h_z\overline{h_{\bar z}}\right)=0.
\end{align*}
Note that $\Phi(\cdot):=K_f^{p-1}(\cdot)$ is positive and continuous. Applying Proposition \ref{phi-laplacian-minimum}, we conclude that $h$ is the minimizer of $\mathscr{E}^{\Phi}_{\mathbb{X}}[\cdot]$ and Theorem \ref{Uniqueness-Phi} further implies that $h$ is unique. For the remaining part of the lemma, we can prove it using the definitions of the two energies and a change of variables.
\end{proof}
\begin{proof}[Proof of Theorem \ref{unique-diffeomorphism}]
     Suppose that $f$ and $\tilde{f}$ are minimizers
\begin{align*}
	\mathscr{E}_{p}[f]=\mathscr{E}_{p}[\tilde{f}]=\inf_{g\in{\rm Diff}_{f_0}(\overline{\mathbb{Y}},\overline{\mathbb{X}})}\mathscr{E}_{p}[g],
\end{align*}
and consider the auxiliary energy functional
\begin{align*}
	\mathscr{E}^{\Phi}_{\mathbb{X}}[\cdot]:=\int_{\mathbb{X}}K_{f}^{p-1}(\cdot)|D(\cdot)|^2dz,
\end{align*}
where $K_{f}$ is defined as in Lemma \ref{minimizer}. By Lemma \ref{minimizer},  $h:=f^{-1}$ is a minimizer of the energy $\mathscr{E}^{\Phi}_{\mathbb{X}}[\cdot]$. In particular,
\begin{align}\label{uniqueness-p-1}
   \mathscr{E}^{\Phi}_{\mathbb{X}}[h]\le\mathscr{E}^{\Phi}_{\mathbb{X}}[\tilde{h}],
\end{align}
where $\tilde{h}=\tilde{f}^{-1}$. Furthermore, Lemma \ref{basic inequality phi-p} asserts that
\begin{align}\label{uniqueness-p-2}
	\mathscr{E}^{\Phi}_{\mathbb{X}}[\tilde{h}]
	\le\mathscr{E}_{p}[f]^{\frac{p-1}{p}}\cdot\mathscr{E}_{p}[\tilde{f}]^{\frac{1}{p}}=\mathscr{E}_{p}[f],
\end{align}
where we use the fact  $\mathscr{E}_{p}[f]=\mathscr{E}_{p}[\tilde{f}]$ in \eqref{uniqueness-p-2}. Combining \eqref{phi=p}, \eqref{uniqueness-p-1} and \eqref{uniqueness-p-2}, we obtain
\begin{align}
	\mathscr{E}^{\Phi}_{\mathbb{X}}[h]\le\mathscr{E}^{\Phi}_{\mathbb{X}}[\tilde{h}]
	\le\mathscr{E}_{p}[f]=\mathscr{E}^{\Phi}_{\mathbb{X}}[h],
\end{align}
which means $h_1$ is also a minimizer of $\mathscr{E}^{\Phi}_{\mathbb{X}}[h]$. Lemma \ref{minimizer} has shown that the minimizer of $\mathscr{E}^{\Phi}_{\mathbb{X}}[\cdot]$ is unique. Hence, we have $f=h^{-1}=\tilde{h}^{-1}=\tilde{f}$, which completes the proof.
\end{proof}
\bibliographystyle{alpha}
\bibliography{Hopf}
\end{document}